\documentclass[11pt]{article}
\usepackage{amsmath}
\usepackage{amsthm}
\usepackage{amssymb}

\theoremstyle{theorem}
\newtheorem{theorem}{Theorem}

\theoremstyle{definition}

\begin{document}

\begin{center}\Large
\textbf{Two Extensions of the Sury's Identity}
\bigskip\large

Ivica Martinjak\\
Faculty of Science, University of Zagreb\\
Bijeni\v cka cesta 32, HR-10000 Zagreb\\
Croatia\\
\end{center}

\noindent
The Fibonacci sequence, denoted by $\{F_n\}_{n \ge 0}$, is usually defined as the second order recurrence relation
\begin{eqnarray}  \label{FiboRec}
F_{n+2}=F_{n+1}+F_{n}, \enspace F_0=0, \enspace F_1=1.
\end{eqnarray}
Its close companions, the Lucas numbers $\{L_n\}_{n \ge 0}$, are defined with the same recurrence formula but with the initial values $L_0=2$ and $L_1=1$. 
Equivalently, these sequences could be defined as the only solutions of the Diophantine equation $x^2-5y^2=4(-1)^n$, $x=L_n$, $y=F_n$, $n \in \mathbb{N}_0$. The most elementary identity encompassing both sequences is 
\begin{eqnarray} \label{FiboLucasElem}
L_n=F_{n-1} + F_{n+1}.
\end{eqnarray}
Two different proofs of the relation
\begin{eqnarray} \label{suryId}
\sum_{k=0}^n 2^k L_k  = 2^{n+1 } F_{n+1}
\end{eqnarray}
were published recently \cite{kwong, sury}. It is also an immediate consequence of (\ref{FiboLucasElem}). We continue the series proving two related Fibonacci-Lucas identities, again using relations (\ref{FiboLucasElem}) and (\ref{FiboRec}).

\begin{theorem} \label{FirstThm}
For the Fibonacci and Lucas sequence
\begin{eqnarray}
\sum_{k=0}^n (-1)^k 2^{n-k} L_{k+1} = (-1)^{n} F_{n+1}.
\end{eqnarray}
\end{theorem}

\begin{proof}
\begin{eqnarray*}
&  2^n&L_1-2^{n-1}L_2+2^{n-2}L_3- \cdots (-1)^nL_{n+1} =\\
&=& 2^n(F_0+F_2)-2^{n-1}(F_1+F_3)+ 2^{n-2}(F_2+F_4)- \cdots (-1)^n(F_n+F_{n+2})\\
&=& 2^n(2F_0+F_1) - 2^{n-1} (2F_1+F_2) + 2^{n-2}(2F_2+F_3)- \cdots (-1)^n(2F_n+F_{n+1})\\
&=& (2^nF_1-2^nF_1) +(-2^{n-1}F_2 + 2^{n-1}F_2) + (2^{n-2}F_3 - 2^{n-2}F_3) + \cdots (-1)^nF_{n+1}\\
&=& (-1)^nF_{n+1}
\end{eqnarray*}
\end{proof}

Theorem \ref{secondThm} is proved in the same fashion. It provides the answer whether there is an identity involving the product $3^nL_n$. Furher generalizations are also possible.

\begin{theorem}  \label{secondThm}
For the Fibonacci and Lucas sequence
\begin{eqnarray} \label{secondThmId}
\sum_{k=0}^n 3^k (L_k + F_{k+1}) = 3^{n+1 } F_{n+1}.
\end{eqnarray}
\end{theorem}

\end{document}